\documentclass[11pt, reqno]{amsart}
\usepackage{lipsum}
\usepackage{amssymb}
\usepackage{amsmath}
\usepackage{amsthm}
\usepackage{tikz}
\usepackage{tikz-cd}
\usepackage{amscd}
\usepackage{hyperref}
\usepackage{enumitem}
\hypersetup{colorlinks,linkcolor={blue},citecolor={blue},urlcolor={red}}
\usepackage{color}
\usepackage[all]{xy}
\theoremstyle{plain}
\usepackage[margin=2.4cm]{geometry}
\numberwithin{equation}{section}

\newcommand{\sch}[1]{\operatorname{{\bf #1}}}

\newtheorem{theorem}{Theorem}[section]

\newtheorem{lemma}[theorem]{Lemma}
\newtheorem{remark}[theorem]{Remark}
\newtheorem{proposition}[theorem]{Proposition}

\newtheorem{definition}[theorem]{Definition}

\begin{document}
\title{Compatibility of Kazhdan and Brauer Homomorphism}
\author{Sabyasachi Dhar}
\maketitle
\begin{abstract}
Let $G$ be a connected split reductive group defined over $\mathbb{Z}$.
Let $F$ and $F'$ be two non-Archimedean $m$-close local fields, where $m$ is a positive integer. D.Kazhdan gave an isomorphism
between the Hecke algebras
${\rm Kaz}_m^F :\mathcal{H}\big(G(F),K_F\big) \rightarrow
\mathcal{H}\big(G(F'),K_{F'}\big)$,
where $K_F$ and $K_{F'}$ are the $m$-th usual congruence subgroups 
of $G(F)$ and $G(F')$ respectively. On the other hand, if $\sigma$ 
is an automorphism of $G$ of prime order $l$, then we have Brauer 
homomorphism ${\rm Br}:\mathcal{H}(G(F),U(F))\rightarrow
\mathcal{H}(G^\sigma(F),U^\sigma(F))$, where $U(F)$ and $U^\sigma(F)$ are compact open subgroups of $G(F)$ and $G^\sigma(F)$ respectively. In this article, we study the compatibility between these two maps in the local base change setting. Further, an application of this compatibility is 
given in the context of linkage--which is the representation theoretic version of Brauer homomorphism.
\end{abstract}
	
\section{Introduction}
The local Langlands correspondence relates the set of irreducible
smooth complex representations of the group of rational points of a reductive group over a local field--with the representation theory of its
Weil--Deligne group. For a connected split reductive group $\textbf{G}$ defined over $\mathbb{Z}$, D. Kazhdan conjectured a link between the local Langlands correspondences for $\sch{G}(F)$ and $\sch{G}(F')$, where $F$ is a $p$-adic field and $F'$ is a non-Archimedean positive characteristic
local field--which are sufficiently close. Kazhdan's approach is via
isomorphisms between the respective Hecke algebras with complex
coefficients, $\mathcal{H}(\sch{G}(F),K)$ and
$\mathcal{H}(\sch{G}(F'),K')$, for some specific choice of
open-compact subgroups $K$ and $K'$ of $\sch{G}(F)$ and $\sch{G}(F')$
respectively (see \cite[Theorem A]{MR874049}). Say $L$ is a 
non-Archimedean local field with residue characteristic $p$ and assume
that $K$ is a nice compact open subgroup of $\sch{G}(L)$--for instance
a congruence subgroup of positive level. When $\sch{G}(L)$ has an
automorphism, denoted by $\sigma$, of prime order $l\neq p$ such that
$\sigma(K)=K$, Treumann--Venkatesh (\cite[Section 4]{MR3432583})
define an $\overline{\mathbb{F}}_l$-algebra morphism between Hecke algebras with $\overline{\mathbb{F}}_l$-coefficients,
$\mathcal{H}(\sch{G}(L), K)^\sigma$ and
$\mathcal{H}(\sch{G}^\sigma(L), K^\sigma)$, which is called the Brauer homomorphism, and it is conjectured to be compatible with Langlands
functoriality. This article aims to study the
compatibility between the Brauer homomorphism and Kazhdan isomorphism in the local base change setting. So far, this type of compatibility has not been addressed in the literature. Using this compatibility, one can relate the functoriality principles for connected reductive algebraic groups defined over the close local fields $F$ and $F'$. This article gives one such treatment in the context of linkage (see \cite[Section 6]{MR3432583})--which is related to the functoriality over $\overline{\mathbb{F}}_l$.
	
Let $\mathfrak{o}_F$ be the ring of integers of $F$, and let
$\mathfrak{p}_F$ be the maximal ideal of $\mathfrak{o}_F$. Let $E$ be
a Galois extension of $F$ of prime degree $l$, where $l$ is different from the residue characteristic of $F$. Let $\sigma$ be a generator of ${\rm Gal}(E/F)$. Similar notations follow for $E$. Let $K_E$ (resp. $K_F$) be the $m$-th congruence subgroup of $\sch{G}(\mathfrak{o}_E)$ (resp. $\sch{G}(\mathfrak{o}_F)$). Note that $K_E$ is stable under
the action of $\sigma$ and let $K_F = K_E\cap \sch{G}(F)$. Treumann--Venkatesh ( \cite[Section 4]{MR3432583}) define Brauer homomorphism, denoted by ${\rm Br}$--which is the $\overline{\mathbb{F}}_l$-algebra homomorphism
$$ {\rm Br}: \mathcal{H}(\sch{G}(E), K_E)^\sigma\rightarrow \mathcal{H}(\sch{G}(F), K_F), $$
obtained by restriction of functions on $\sch{G}(E)$ to
$\sch{G}(F)$. On the other hand, following Kazhdan's approach in \cite[Theorem A]{MR874049}, we get an $\overline{\mathbb{F}}_l$-algebra isomorphism
(Proposition \ref{Kaz_isom_mod_l})
$$
{\rm Kaz}_m^F : \mathcal{H}(\sch{G}(F),K_F)
\rightarrow\mathcal{H}(\sch{G}(F'),K_{F'}),
$$
provided the fields $F$ and $F'$ are sufficiently close. The above isomorphism ${\rm Kaz}_m^F$ was originally constructed for complex Hecke algebras (Proposition \ref{Kaz_isom_complex})--which relies on the finiteness of complex Hecke algebras, due to Bernstein (\cite[Corollary 3.4]{Bernstein_center}). For arbitrary Noetherian $\mathbb{Z}_l$-algebra $R$, where $\mathbb{Z}_l$ is the ring of integers of the field of $l$-adic numbers, the finiteness of Hecke algebras with coefficients in $R$ follows from the work of \cite[Theorem 1.1]{dat2024finiteness}. It enables us to define the isomorphism ${\rm Kaz}_m^F$ between mod-$l$ Hecke algebras. 

The method of close local fields is fruitful in obtaining analogous results for reductive groups defined over local fields of positive characteristic--which are known for reductive groups over $p$-adic fields. Using the close local fields approach, A.I. Badulescu (\cite{J_L_Badulescu}) proved the Jacquet-Langlands correspondence for the general linear group ${\rm GL}_n(F')$. A similar technique was used by Badulescu--Henniart--Lemaire--S\'echerre (\cite{Unitary_dual_GLn(D)}) to classify the smooth unitary dual of ${\rm GL}_n(D)$, where $D$ is a central division algebra over $F'$. Kazhdan isomorphism and its various properties have been further studied by several others, notably by R. Ganapathy (see \cite{ganapathy2015local}, \cite{Radhika_Hecke_algebra_general}).
	
We now state the main results of this article. Let $\textbf{G}$ be a connected split reductive group defined over $\mathbb{Z}$. Let $F$ and $F'$ be two non-Archimedean local fields with the same residue characteristic $p$. We assume that $F$ and $F'$ are $m$-close (i.e., there exists a ring isomorphism $\mathfrak{o}_F/\mathfrak{p}_F^m\xrightarrow{\sim} \mathfrak{o}_{F'}/\mathfrak{p}_{F'}^m$). Let $E/F$ be a finite Galois extension of prime degree $l$ with $l\ne p$. Then there exists a finite Galois extension $E'/F'$ of degree $l$ such that $E'$ is $em$-close to $E$, where $e=1$ if $E/F$ is unramified, and $e=l$ if $E/F$ is totally ramified (Lemma \ref{lemma_close}). Let $\sigma$ (resp. $\sigma'$) be a generator
of the group ${\rm Gal}(E/F)$ (resp. ${\rm Gal}(E'/F')$). Then we have the Kazhdan isomorphism ${\rm Kaz}_{em}^E$ between the mod-$l$ Hecke algebras $\mathcal{H}(\textbf{G}(E),K_E)$ and $\mathcal{H}(\textbf{G}(E'),K_{E'})$, and it is equivariant under the Galois actions (Proposition \ref{comp_galois_action})--which leads to the following diagram: 
\begin{equation}\label{c_1} 
\xymatrix{
	\mathcal{H}(\textbf{G}(E),K_E)^\sigma \ar[dd]_{{\rm Br}}
	\ar[rr]^{{\rm Kaz}_{em}^E}
	&&  \mathcal{H}(\textbf{G}(E'),K_{E'})^{\sigma'} \ar[dd]^{{\rm Br}'} \\\\
	\mathcal{H}(\textbf{G}(F),K_F) \ar[rr]_{{\rm Kaz}_m^F} &&
	\mathcal{H}(\textbf{G}(F'),K_{F'})} 
\end{equation} 
In this article, we address the natural question about the commutativity of the above diagram. Let us state it as a theorem.
\begin{theorem}\label{intro_thm}
Let $F$ and $F'$ be two non-Archimedean $m$-close local fields with residue characteristic $p$. Let $E$ be a finite Galois extension of $F$ of prime degree $l$ with $l\ne p$, and let $E'$ be the Galois extension of $F'$, as indicated above. Then, for any connected split reductive group $\textbf{G}$ defined over $\mathbb{Z}$, we have
$$
{\rm Br}'\circ{\rm Kaz}_{em}^E = {\rm Kaz}_m^F\circ{\rm Br}.
$$  
\end{theorem}
Using the above theorem, we prove the compatibility of linkage with Kazhdan isomorphism. Say $(\pi,V)$ is an irreducible smooth $\overline{\mathbb{F}}_l$-representation of $\textbf{G}(E)$ such that $\pi$ is isomorphic to twisted representation $\pi^\sigma$. Then we have the Tate cohomology groups $\widehat{H}^i(\sigma,V)$, for $i\in\{0,1\}$, defined as $\overline{\mathbb{F}}_l$-representations of $\textbf{G}(F)$. An irreducible sub--quotient $\rho$ of $\widehat{H}^i(\sigma,V)$ is said to be linked with $\pi$. Let $\rho'$ (resp. $\pi'$) be the smooth irreducible $\overline{\mathbb{F}}_l$-representation of $\textbf{G}(F')$ (resp. $\textbf{G}(E')$) associated with $\rho$ (resp. $\pi$) via Kazhdan isomorphism and the natural correspondence between the set of simple modules over Hecke algebras and the set of irreducible representations of locally profinite group (see \cite[Section 2.3]{ganapathy2015local}). Then, using Theorem \ref{intro_thm} and the compatibility of Brauer homomorphism with linkage (see subsection \ref{com_linkage}), we show that $\rho'$ is linked with $\pi'$ if and only if $\rho$ is linked with $\pi$ (Theorem \ref{linkage_kaz}).
	
We now briefly explain the contents of this article. Section $2$ reviews the Hecke algebra structure associated with any locally profinite group. In Section $3$, we discuss the Brauer homomorphism. In Section $4$, we review the Kazdhan isomorphism and formulate its mod-$l$ version. We also set up some initial results, which are crucial for the main result. Theorem \ref{intro_thm} is proved in Section $5$. In Section $6$, we prove the compatibility of linkage with the Kazhdan isomorphism. It also includes some finiteness results on the Tate cohomology groups.
	
\section{Hecke Algebra}\label{HA}
This section reviews the Hecke algebra structure on any locally profinite group. We refer to \cite[Subsection 2.10]{MR3432583} for more insight.
\subsection{}\label{id}
Let $k$ be an algebraically closed field of positive characteristic $l$. Let $G$ be a locally profinite (i.e., locally compact and totally disconnected) group. Let $K$ be a compact open subgroup of $G$. We denote by $G/K$ the set of all distinct left cosets of $K$ in $G$, and it is a discrete set with a left action of $G$. For any $g\in G$, we sometimes use the notation $[g]$ to denote the left coset $gK$. There is a natural action of $G\times G$, via left translation, on the space of $k$-valued functions on the discrete set $G/K\times G/K$. We denote by $\mathcal{F}(G//K)$ the space of all such $k$-valued functions on $G/K\times G/K$ that are invariant under the action of the diagonal subgroup $\Delta G=\{(g,g):g\in G\}$, and whose support is a union of finitely many $G$-orbits. There is a multiplicative structure in the space
$\mathcal{F}(G//K)$, given by
$$
(f_1 * f_2)([x],[z])=\sum_{y\in G/K}f_1([x],[y])f_2([y],[z]),
$$
for all $[x],[z]\in G/K$ and $f_1,f_2\in\mathcal{F}(G//K)$. The space $\mathcal{F}(G//K)$ with the above multiplicative structure is a $k$-algebra, and it is called the {\it Hecke} algebra of $G$ associated with $K$. We now give an equivalent description of the Hecke algebra relevant to our context.  
\subsection{}
Let us consider the equivalence relation $\sim$ on the set $G/K\times G/K$, where
$$
([x],[y]) \sim ([x'],[y'])
$$ 
if and only if there exists $g\in G$ such that
$[gx] = [gx']$ and $[gy] = [gy']$. Then for any $[g],[h]\in G/K$, the map $([g],[h])\mapsto Kg^{-1}hK$ induces a bijection from the set of equivalence classes $(G/K\times G/K)/\sim$ onto the double coset space $K\setminus G/K$. 
	
Let $\mathcal{H}(G,K)$ be the space of all
$k$-valued functions which are bi-$K$-invariant i.e., $f(k_1gk_2) = f(g),$ for all $k_1,k_2 \in K$, $g \in G$ and $f \in \mathcal{H}(G,K)$. Then, the preceding bijection gives the following map
$$ 
\mathcal{F}(G//K) \longrightarrow \mathcal{H}(G,K)
$$
\begin{equation}\label{bij_2}
\phi \longmapsto \widetilde{\phi},
\end{equation} 
where the function $\widetilde{\phi}:G\rightarrow k$ is defined as $\widetilde{\phi}(g):=\phi([1],[g])$, for all $g\in G$. The map (\ref{bij_2}) is a bijection, and it induces a $k$-algebra structure on the space $\mathcal{H}(G,K)$. In the literature, many authors refer to the $k$-algebra $\mathcal{H}(G,K)$ as the Hecke algebra of $G$ associated with $K$. 
	
\section{Brauer Homomorphism}\label{B}
\subsection{}
Let $G$ be a locally profinite group, and let $\sigma$ be an automorphism of $G$ of order $l$. Let $K$ be a compact open subgroup of $G$ with $\sigma(K)=K$. Let $G^\sigma$ (resp. $K^\sigma$) be the the fixed subgroup of $G$ (resp. $K$) under the action of $\sigma$. Moreover, the compact open subgroup $K$ is called {\it $\sigma$-plain} if
\begin{enumerate}
\item The inclusion 
$$G^\sigma/K^\sigma \hookrightarrow (G/K)^\sigma$$
$$ gK^\sigma\mapsto gK $$
induces a bijection.
\item There exists a subgroup $K'$ of $K$ of finite index such that $K'$
is an inverse limit of finite subgroups, each of which is coprime to $l$.
\end{enumerate}
\subsection{}
For any $\sigma$-plain compact open subgroup $K$ of $G$, the discrete set $G/K \times G/K$ is equipped with an action of the group $\langle \sigma\rangle$, defined as
$$
\sigma.([g],[h]) := \big([\sigma(g)],[\sigma(h)]\big),
$$ 
for all $[g],[h]\in G/K$. This action further factorizes through the  quotient space $G/K\times G/K/\sim$, and it is given by
$$
\sigma.\overline{([g],[h])} =
\overline{([\sigma(g)],[\sigma(h)])},
$$
where $\overline{([g],[h])}$ denotes the equivalence class of the element $([g],[h])$. The space
$\mathcal{F}(G//K)$ carries a natural action of $\langle\sigma\rangle$, defined by
$$
(\sigma.\phi)([x],[y]) =
\phi\big([\sigma^{-1}(x)],[\sigma^{-1}
(y)]\big),
$$ 
for all $x,y\in G$ and $\phi\in \mathcal{F}(G//K)$. In the article
\cite[Section 4]{MR3432583}, the authors proved that the Brauer homomorphism ${\rm Br}$ is a $k$-algebra homomorphism from $\mathcal{F}(G//K)^\sigma$ to
$\mathcal{F}(G^\sigma//K^\sigma)$, and defined by
the restriction
\begin{equation}\label{Br_hom}
{\rm Br}(\phi):= {\rm Res}_{(G^\sigma/K^\sigma \times
G^\sigma/K^\sigma)}(\phi)
\end{equation} for all $\phi\in \mathcal{F}(G//K)^\sigma$.
	
\subsection{}
Now, we formulate the Brauer homomorphism (\ref{Br_hom}) for the $k$-algebra $\mathcal{H}(G,K)$, defined as in the preceding section. First, let us consider the action of the group $\langle\sigma\rangle$
on the double coset space $K\setminus G/K$ :
$$ 
\sigma.(KgK) := K\sigma(g)K,
$$ 
for all $g\in G$. Then, we have a natural
induced action of $\langle\sigma\rangle$ on the space $\mathcal{H}(G,K)$, defined as
$$
(\sigma.f)(g):= f(\sigma^{-1}(g)),
$$ 
for all $g\in G$. Note that the
bijection (\ref{bij_2}) preserves the action of the group
$\langle\sigma\rangle$. Then, we have a commutative diagram
$$ 
\xymatrix{
	\mathcal{F}(G//K)^\sigma \ar[dd]_{{\rm Br}}  \ar[rr]^{}  &&
	\mathcal{H}(G,K)^\sigma \ar[dd]^{{\rm \overline{Br}}} \\\\
	\mathcal{F}(G^\sigma//K^\sigma) \ar[rr]_{} &&  \mathcal{H}(G^\sigma,K^\sigma)
}
$$ 
Here, the horizontal arrows in the diagram are the bijections induced by (\ref{bij_2}). We observe that the $k$-algebra map ${\rm \overline{Br}}:\mathcal{H}(G,K)^\sigma
\rightarrow\mathcal{H}(G^\sigma,K^\sigma)$, induced by the above diagram, is the restriction map to subgroup $G^\sigma$, i.e.,
$$
{\rm \overline{Br}}(f):= {\rm Res}_{G^\sigma}(f),
$$ 
for all $f \in \mathcal{H}(G,K)^\sigma$
	
\section{Kazhdan isomorphism}\label{K}
In this part, we review Kazhdan isomorphism and set up some initial results. Let $F$ and $F'$ be the non-Archimedean local fields, both of residue characteristic $p$, with ring of
integers $\mathfrak{o}_F$ and $\mathfrak{o}_{F'}$ and its maximal ideals
$\mathfrak{p}_F$ and $\mathfrak{p}_{F'}$ respectively. The fields $F$
and $F'$ are called {\it $m$-close} if there is a ring isomorphism
$\Lambda:
\mathfrak{o}_F/\mathfrak{p}_F^m\rightarrow\mathfrak{o}_{F'}/
\mathfrak{p}_{F'}^m$. We choose uniformizer $\varpi$ (resp. 
$\varpi'$) of the field $F$ (resp. $F'$) such that the class of $\varpi$ corresponds to the class of $\varpi'$ via the map $\Lambda$.
	
\subsection{}
Let $\textbf{G}$ be a split connected reductive group defined over
$\mathbb{Z}$. Fix a Borel subgroup $\textbf{B}$ of $\textbf{G}$,
defined over $\mathbb{Z}$. Let $\textbf{B}=\textbf{T}\textbf{U}$ be a
Levi decomposition over $\mathbb{Z}$ with maximal torus $\textbf{T}$
and unipotent radical $\textbf{U}$. We denote by $X^*(\textbf{T})$ and
$X_*(\textbf{T})$ the character lattice and the co-character lattice of
$\textbf{T}$ respectively. Then, there is a natural
$\mathbb{Z}$-bilinear pairing
$$ X^*(\textbf{T})\times X_*(\textbf{T})\longrightarrow\mathbb{Z} $$
$$ (\alpha,\lambda)\longmapsto\langle\alpha,\lambda\rangle, $$
where $(\lambda\circ\alpha)(x)=x^{\langle\alpha,\lambda\rangle}$, for
all $x\in\mathbb{G}_m$.  Let $\Phi\subseteq X^*(\textbf{T})$ be the
set of roots of $\textbf{T}$ in $\textbf{G}$, and let $\Phi^+$ be the
set of positive roots of $\textbf{T}$ in $\textbf{B}$. 
\subsubsection{}
Let $G_F$ denotes the group $\textbf{G}(F)$. Consider the subgroup
$$ K_F={\rm Ker}\big(\textbf{G}(\mathfrak{o}_F)
\longrightarrow\textbf{G}(\mathfrak{o}_F
/\mathfrak{p}_F^m)\big). $$ 
Then $K_F$ is a compact open subgroup of
$G_F$. Fix a Haar measure $\mu_F$ on $G_F$ such that
$\mu_F(K_F)=1$. For each $g\in G_F$, we denote by $t_g$ the
characteristic function on the double coset $K_FgK_F$. Then the Hecke
algebra $\mathcal{H}(G_F,K_F)$ is generated as a $k$-vector space by
$\{t_g:g\in G_F\}$. We denote by $\mathcal{H}_{\mathbb{Z}}(G_F,K_F)$ the space generated by $\{t_g:g\in G_F\}$ over $\mathbb{Z}$. Set
$$ X_*(\textbf{T})^{-} = \{\lambda\in X_*(\textbf{T}):\langle\alpha,\lambda
\rangle\leq 0,\alpha\in \Phi^{+}\}. $$ 
We have the Cartan decomposition
\begin{equation}\label{CD}
G_F=\coprod_{\mu\in X_*(\textbf{T})^{-}}\textbf{G}(\mathfrak{o}_F)\varpi_{\mu}
\textbf{G}(\mathfrak{o}_F),
\end{equation}
where $\varpi_{\mu}=\mu(\varpi)$. For each $\mu\in X_*(\textbf{T})^{-}$, the	double coset $G_F(\mu):=\textbf{G}(\mathfrak{o}_F)\varpi_{\mu}
\textbf{G}(\mathfrak{o}_F)$ is a homogeneous space under the action of the
group $\textbf{G}(\mathfrak{o}_F)\times\textbf{G}(\mathfrak{o}_F)$, defined as
$$ (x,y).g:=xgy^{-1}, $$ 
for all $x,y\in \textbf{G}(\mathfrak{o}_F)$ and $g\in
G_F(\mu)$. Then, the finite group
$H_m:=\textbf{G}(\mathfrak{o}_F/\mathfrak{p}_F^m)\times
\textbf{G}(\mathfrak{o}_F/\mathfrak{p}_F^m)$ acts transitively on the
set $\{K_FxK_F:x \in G_F(\mu)\}$. Let $\Gamma_{\mu}(F)$ be the
stabilizer of $K_F\varpi_{\mu}K_F$ in $H_m$. Let
$T_\mu(F)\subseteq\textbf{G}(\mathfrak{o}_F)
\times\textbf{G}(\mathfrak{o}_F)$
be the set of representatives of the coset space $H_m/\Gamma_\mu(F)$
under the composition
$$ \textbf{G}(\mathfrak{o}_F)\times\textbf{G}
(\mathfrak{o}_F)\xrightarrow{{\rm mod}-\mathfrak{p}_F^m}H_m\longrightarrow H_m/\Gamma_\mu(F) $$
Then (\ref{CD}) can be re-written as
\begin{equation}\label{MCD}
G_F = \coprod_{\mu\in X_*(\textbf{T})^{-}}\coprod_{(b_i,b_j)\in
T_\mu(F)}K_Fb_i\varpi_\mu b_j^{-1}K_F.
\end{equation}
\subsubsection{}
For each $\mu\in X_*(\textbf{T})^{-}$, the isomorphism $\Lambda$ induces  a bijection $T_\mu(F)\rightarrow T_\mu(F')$. Following Kazhdan's constructions in \cite[Section 1]{MR874049}, we get an isomorphism of $\mathbb{Z}$-modules:
\begin{equation}\label{Kaz_isom_integral}
\mathcal{H}_{\mathbb{Z}}(G_F,K_F)\longrightarrow \mathcal{H}_{\mathbb{Z}}(G_{F'},K_{F'}) 
\end{equation}
$$ t_{a_i\varpi_{\lambda}a_j^{-1}}\longmapsto
t_{a_i'\varpi_{\lambda}'a_j'^{-1}}, $$
and this induces a $\mathbb{C}$-vector space isomorphism
$$ {\rm Kaz}_{m,\mathbb{C}}^F : \mathcal{H}_{\mathbb{C}}(G_F,K_F)\longrightarrow
\mathcal{H}_{\mathbb{C}}(G_{F'},K_{F'}), $$
where $\lambda\in X_*(\textbf{T})^{-}$, $(a_i,a_j)\in T_\lambda(F)$ and
$(a_i',a_j')$ is the corresponding element in $T_\lambda(F')$. The following result is due to Kazhdan (\cite[Theorem A]{MR874049}).
\begin{proposition}\label{Kaz_isom_complex}
\rm Given $m \geq 1$. There exists a positive integer $r \geq m$ such that if $F$ and $F'$ are $r$-close, the map ${\rm Kaz}_{m,\mathbb{C}}^F$, defined above, is infact an isomorphism of $\mathbb{C}$-algebras.
\end{proposition}
In order to establish that the vector space isomorphism ${\rm Kaz}_{m,\mathbb{C}}^F$ is compatible with algebra structures, the fields need to be a few levels closer. Kazhdan conjectured that one could take $r=m$ in Proposition \ref{Kaz_isom_complex}--which was proved by R. Ganapathy (\cite[Corollary 3.15]{ganapathy2015local}).
	
\subsection{Kazhdan isomorphism over positive characteristic}
In this subsection, we formulate a version of Proposition \ref{Kaz_isom_complex} for mod-$l$ Hecke algebras. Recall that $\mathcal{H}(G_F,K_F)$ denotes the Hecke algebra, consisting of all $k$-valued smooth, compactly supported and bi-$K_F$-invariant functions on $G_F$. Suppose $F$ and $F'$ are $m$-close. We denote by $X_m(F)$ the set of all $K_F$ double cosets in $G_F$. Similar notations are followed for $F'$. The bijection $\phi :X_m(F)\rightarrow X_m(F')$ gives the $\mathbb{Z}$-module isomorphism (see \cite[Section 1]{MR874049}) 
of integral Hecke algebras 
$$ \mathcal{H}_{\mathbb{Z}}(G_F,K_F) \longrightarrow \mathcal{H}_{\mathbb{Z}}(G_{F'},K_{F'}) $$
$$ t_x \longmapsto t_{\phi(x)}, $$
which induces an isomorphism of $k$-vector spaces
$$  {\rm Kaz}_m^F : \mathcal{H}(G_F,K_F)\longrightarrow
\mathcal{H}(G_{F'},K_{F'}). $$
Kazhdan's approach in proving the Hecke algebra isomorphism over $\mathbb{C}$, relies on the fact that $\mathcal{H}_{\mathbb{C}}(G_F,K_F)$ is finitely presented--which 
follows from the work of \cite{Bernstein_center}. So, in order to establish that ${\rm Kaz}_m^F$ is an algebra map, we need the 
finiteness of Hecke algebras over $k$--which is made available 
due to seminal work of Dat--Helm--Kurinczuk--Moss (\cite{dat2024finiteness}).
\begin{lemma}\label{Hecke_finiteness}
For any compact open subgroup $K$ of $G_F$, the Hecke algebra $\mathcal{H}(G_F,K)$ is finitely presented.
\end{lemma}
\begin{proof}
Note that an algebra $\mathcal{A}$ is finitely presented if $\mathcal{A}$ is finitely generated and Noetherian. Then the 
above lemma follows from \cite[Theorem 1.1 and Corollary 1.4]{dat2024finiteness}.
\end{proof}
\begin{remark}\label{rmk_alg_hom}
\rm We now recall some results from \cite{MR874049}. Let $C$ be a 
finite subset of $X_*(\textbf{T})^{-}$ and let 
$$ G_C = \bigcup_{\mu\in C}G_F(\mu). $$
By \cite[Lemma 1.3]{MR874049}, there exists a positive integer 
$n_C \geq m$, such that for all $g\in G_C$, we have 
$$ gK_{n_C}(F)g^{-1} \subseteq K_F, $$ 
where $K_{n_C}(F)$ is the $n_C$-th usual congruence subgroup of $G_F$. Moreover, if $F$ and $F'$ are $n_C$-close, then
\begin{equation}\label{kaz_product_level}
{\rm Kaz}_m^F(f_1*f_2) = {\rm Kaz}_m^F(f_1) * {\rm Kaz}_m^F(f_2),
\end{equation} 
for all $f_1,f_2\in \mathcal{H}(G_F,K_F)$ supported on $G_C$.
\end{remark} 
The following lemma, which was essential for 
\cite[Theorem A]{MR874049}, is valid over $k$ as well. 
\begin{lemma}\label{lemma_conv}
Let $C$ be a finite subset of $X_*(\textbf{T})^{-}$. Then 
\begin{enumerate}
	\item for any $\lambda, \mu \in X_*(\textbf{T})^{-}$, we have
	$$ t_{\lambda(\varpi)} * t_{\mu(\varpi)} = 
	t_{(\lambda+\mu)(\varpi)} $$
	\item for all $x_1,x_2\in G_C$, we have
	$$ t_{x_1\lambda(\varpi) x_2} = t_{x_1}*t_{\lambda(\varpi)}*
	t_{x_2}. $$
\end{enumerate}
\end{lemma}
\begin{proof}
It follows from \cite[Theorem 2]{MR0304563} by extension of scalars.
\end{proof}
From these, we obtain Kazhdan isomorphism of mod-$l$ Hecke algebras. The proof exactly follows the same line of arguments of 
\cite[Theorem A]{MR874049}.
\begin{theorem}\label{Kaz_isom_mod_l}
Given an integer $m \geq 1$. There exists a positive integer 
$r \geq m$ such that if $F$ and $F'$ are $r$-close, the map 
${\rm Kaz}_m^F$ is an isomorphism of $k$-algebras.
\end{theorem}
\begin{proof}
We will show that ${\rm Kaz}_m^F$ preserves the ring structures. 
Fix a finite subset $C\subseteq X_*(\textbf{T})^{-}$ 
containing $0$, such that $C$ generates $X_*(\textbf{T})^{-}$ as 
a semigroup. Consider the finite set 
$$ X_0 = \bigcup_{\lambda \in C} X_\lambda, $$
where $X_\lambda =\big\{K_FgK_F \mid g\in  \textbf{G}(\mathfrak{o}_F)\lambda(\varpi) \textbf{G}(\mathfrak{o}_F)\big\}$. 
It follows from Lemma \ref{lemma_conv} that the algebra $\mathcal{H}(G_F,K_F)$ is generated by the set $\big\{t_g:g\in X_0\big\}$, where $t_g$ denotes the characteristic function on the 
double coset $K_FgK_F$. Assume that $X_0 = 
\{x_1,\dots, x_s\}$. Let $k\langle X_1,\dots,X_s\rangle$ denote 
the non-commutative polynomial algebra (free algebra) with generators $X_1,\cdots,X_s$. Since the Hecke algebra $\mathcal{H}(G_F,K_F)$ is finitely presented, there exists a surjective $k$-algebra homomorphism
$$ \varphi : k\langle X_1\dots,X_s\rangle \rightarrow \mathcal{H}(G_F,K_F) $$
$$ X_i \longmapsto t_{x_i} $$
such that ${\rm Ker}(\varphi)$ is a two-sided ideal generated by non-commutative polynomials $f_1, f_2,\dots, f_d \in k\langle X_1,\dots,X_s\rangle$. Let 
$$ M = {\rm max}\big\{a_i : 1\leq i\leq d\big\}, $$ 
where $a_i$ is the degree of the polynomial $f_i$. Then there exists 
a finite set $B\subseteq \Phi^+$ such that 
$$ X_B = \bigcup_{\lambda\in B} X_\lambda $$ 
contains a compact set of the form $\big\{g_1g_2\dots g_M: g_i\in X_0\big\}$. By Remark \ref{rmk_alg_hom}, there is an integer 
$r=r_B\geq m$, such that if $F$ and $F'$ are $r$-close, we have 
$$ {\rm Kaz}_m^F(f_1*f_2) = {\rm Kaz}_m^F(f_1)*{\rm Kaz}_m^F(f_2), $$
for all functions $f_1, f_2$ supported on $G_B$. Now, for each $1\leq i\leq d$, the above identity gives 
$$ f_i\big({\rm Kaz}_m^F(t_{x_1}),\dots, 
{\rm Kaz}_m^F(t_{x_s})\big) = 0 $$
or, equivalently
$$ f_i\big(t_{\phi(x_1)},\dots,t_{\phi(x_s)}\big) = 0. $$
Consider the homomorphism 
$$ \varphi':k\langle X_1,\dots,X_s\rangle \longrightarrow \mathcal{H}(G_{F'},K_{F'}) $$
$$ X_j \longmapsto t_{\phi(x_j)}. $$
Note that $\varphi'$ factors through ${\rm Ker}(\varphi)$, and 
hence we get a unique algebra morphism
$$ \widetilde{\varphi'} : \mathcal{H}(G_F,K_F)\longrightarrow \mathcal{H}(G_{F'},K_{F'}) $$
such that $\widetilde{\varphi'}(t_{x_j}) = t_{\phi(x_j)}$, 
for each $j$. Since $C$ generates $X_*(\textbf{T})^{-}$ as a semigroup, 
using the identities in Lemma \ref{lemma_conv}, we get that 
$$ \widetilde{\varphi'}(t_x) = t_{\phi(x)} = {\rm Kaz}_m^F(t_x), $$
for all $x\in X_m(F)$. This shows that $\widetilde{\varphi'} = {\rm Kaz}_m^F$, which is an algebra morphism. Hence, the theorem.
\end{proof}
	
\subsection{Galois extension of close local fields}\label{galois_comp}
Let $F$ and $F'$ be two non-Archimedean $m$-close local fields with residue characteristic $p$. Deligne (\cite[Section 3.5]{Deligne_Galois_rep}) proved that there is an isomorphism 
$$ 
{\rm Gal}(\overline{F}/F)/\mathcal{I}_F^m\xrightarrow{{\rm Del}_m}{\rm Gal}(\overline{F'}/F')/\mathcal{I}_{F'}^m,
$$
where $\overline{F}$ (resp. $\overline{F'}$) is the seperable closure of $F$ (resp. $F'$), and $\mathcal{I}_F^m$ (resp. $\mathcal{I}_{F'}^m$) is the $m$-th higher ramification subgroup of the inertia group $\mathcal{I}_F$ (resp. $\mathcal{I}_{F'}$). Let $E$ be a finite Galois extension of $F$ of prime degree $l$ with $l\ne p$. Following the description of the isomorphism ${\rm Del}_m$ (see \cite[Section 1.3]{Deligne_Galois_rep} and \cite[Section 2.1]{ganapathy2015local}), we have a finite Galois extension $E'/F'$ of same degree $l$. Moreover, the extension $E'/F'$ can be so constructed that the fields $E$ and $E'$ are also sufficiently close.
\begin{lemma}\label{lemma_close}
Let $F$ and $F'$ be two $m$-close non-Archimedean local fields with residue characteristic $p$. Let $E$ be a finite Galois extension of $F$ with $[E:F]=l$, where $l$ and $p$ are distinct primes. Then there exists a finite Galois extension $E'$ of $F'$ of degree $l$ such that
$E'$ is $m$-close to $E$ if $E/F$ is unramified, and $E'$ is $lm$-close to $E$ if $E/F$ is totally ramified.
\end{lemma}	
\begin{proof}
The lemma follows from \cite{Deligne_Galois_rep}. For completeness, we give a sketch of the proof here. The proof is in two parts. In the first part, we consider 
the situation where $E$ is unramified over $F$. The second part deals with totally ramified case. Note that the extension $E/F$ is tamely ramified as $l\ne p$. We denote by $A_m$ and $A_m'$ the quotient rings $\mathfrak{o}_F/\mathfrak{p}_F^m$ and $\mathfrak{o}_{F'}/\mathfrak{p}_{F'}^m$ respectively. Recall that the isomorphism $\Lambda : A_m\xrightarrow{\sim} A_m'$ takes the class of $\varpi$ to that of $\varpi'$. 
\subsubsection{Unramified case}\label{Galois_unramified}
Suppose $E/F$ is unramified. In this case, the field $E$ is of the form $F[T]/(f)$, where $f$ is a monic irreducible polynomial with coefficients in $\mathfrak{o}_F$ such that its ${\rm mod}-\mathfrak{p}_F$ reduction, say $\overline{f}$, is irreducible. Moreover, the quotient ring $\mathfrak{o}_E/\mathfrak{p}_E^m$ is isomorphic to $A_m[T]/(\overline{f})$. Let $\overline{f}'$ be the image of $\overline{f}$ under the ring isomorphism $\Lambda$. Then we have  
$$
\mathfrak{o}_E/\mathfrak{p}_E^m\xrightarrow{\sim}
A_m'[T]/(\overline{f}').
$$
Let $f'$ be an irreducible polynomial in $\mathfrak{o}_{F'}[T]$ with its ${\rm mod}-\mathfrak{p}_F^m$ reduction $\overline{f}'$. Let $E'$ be the field extension $F'[T]/(f')$. Then $E'$ is unramified over $F'$, and there is a ring isomorphism
$$
\mathfrak{o}_{E'}/\mathfrak{p}_{E'}^m\simeq A_m'[T]/(\overline{f}').
$$
Thus we have the ring isomorphism  $\mathfrak{o}_E/\mathfrak{p}_E^m\xrightarrow{\sim}\mathfrak{o}_{E'}/\mathfrak{p}_{E'}^m$, which implies that the fields $E$ and $E'$ are $m$-close.
\subsubsection{Totally ramified case}\label{Galois_Totally_ram} 
Suppose $E$ is totally ramified over $F$. Here the field $E$ is of the form $F[T]/(T^l-\varpi)$ and its ring of integers  $\mathfrak{o}_E$ is given by $\mathfrak{o}_F[T]/(T^l-\varpi)$. Recall that the class of $\varpi$ corresponds to the class of $\varpi'$ via the isomorphism $\Lambda$. Now consider the field extension, 
$$
E':=F[T]/(T^l-\varpi')
$$ 
Then $E'$ is totally ramified over $F'$ with the ring of integers $\mathfrak{o}_{E'}=\mathfrak{o}_{F'}[T]/(T^l-\varpi')$. 
Let $\overline{\varpi}$ be the image of $\varpi$ under ${\rm mod}-\mathfrak{p}_F^m$. Then, the natural surjection 
$$
\mathfrak{o}_F[T]/(T^l-\varpi)\longrightarrow A_m[T]/
(T^l-\overline{\varpi})
$$ 
factorizes through the ideal $\mathfrak{p}_F^m[T]/(T^l-\varpi)$, 
and it gives the ring isomorphism
$$
\mathfrak{o}_E/\mathfrak{p}_E^{lm}\xrightarrow{\sim}
A_m[T]/(T^l-\overline
{\varpi}).
$$ 
Similarly, there is a ring isomorphism,
$$
\mathfrak{o}_{E'}/\mathfrak{p}_{E'}^{lm}
\simeq A_m'[T]/(T^l-\overline{\varpi'}),
$$
where $\overline{\varpi'}$ is the image of $\varpi'$ under the ${\rm mod}-\mathfrak{p}_{F'}^m$ reduction map. On the other hand, the map $\Lambda$ induces the ring isomorphism 
$$
A_m[T]/(T^l-\overline{\varpi})\xrightarrow{\sim} A_m'[T]/
(T^l-\overline{\varpi'}).
$$ 
This shows that the ring   
$\mathfrak{o}_E/\mathfrak{p}_E^{lm}$ is isomorphic to
$\mathfrak{o}_{E'}/\mathfrak{p}_{E'}^{lm}$, and hence the fields $E$ and $E'$ are $lm$-close. 
\end{proof}
\begin{remark}
\rm Let $F$, $F'$, $E$ and $E'$ be as in Lemma \ref{lemma_close}. We simply write $E$ and $E'$ are $em$-close, where we define $e=1$ if $E/F$ is unramified and $e=l$ if $E/F$ is totally ramified. We fix a ring isomorphism $\Pi:\mathfrak{o}_E/\mathfrak{p}_E^{em}\rightarrow\mathfrak{o}_{E'}/\mathfrak{p}_{E'}^{em}$ as above, induced by $\Lambda$, which makes commutative the diagram 
$$
\xymatrix{
	\mathfrak{o}_F/\mathfrak{p}_F^m \ar[dd]_{}  \ar[rr]^{\Lambda}  &&
	\mathfrak{o}_{F'}/\mathfrak{p}_{F'}^m \ar[dd]^{} \\\\
	\mathfrak{o}_E/\mathfrak{p}_E^{em}  \ar[rr]_{\Pi} &&
	\mathfrak{o}_{E'}/\mathfrak{p}_{E'}^{em}}
$$
Here, the vertical arrows are the natural inclusion maps. We choose uniformizer $\pi$ (resp. $\pi'$) of the field $E$ (resp. $E'$) such that the class of $\pi$ is mapped to the class of $\pi'$ under the isomorphism $\Pi$. For any $\overline{a}\in\mathfrak{o}_F/\mathfrak{p}_F^m$ (or, $\mathfrak{o}_E/\mathfrak{p}_E^{em}$), we write $\overline{a}'$ for the image $\Lambda(\overline{a})$ (or, $\Pi(\overline{a})$), whenever no confusion arises.
\end{remark}
	
\begin{remark}\label{Galois_isom}
\rm Let $\Gamma$ and $\Gamma'$ denote the Galois groups ${\rm Gal}(E/F)$ and ${\rm Gal}(E'/F')$ respectively. We choose generators $\sigma$ and $\sigma'$ of $\Gamma$ and $\Gamma'$ respectively, which makes commutative the following diagram 
$$
\xymatrix{
	\mathfrak{o}_E/\mathfrak{p}^{em}_E \ar[dd]_{\Pi}  \ar[rr]^{\sigma}  &&
	\mathfrak{o}_E/\mathfrak{p}^{em}_E \ar[dd]^{\Pi} \\\\
	\mathfrak{o}_{E'}/\mathfrak{p}^{em}_{E'}  \ar[rr]_{\sigma'} &&
	\mathfrak{o}_{E'}/\mathfrak{p}^{em}_{E'}
}
$$
We now describe the generators $\sigma$ and $\sigma'$ in totally ramified case. For the unramified case, one similarly chooses $\sigma$ and $\sigma'$. So, assume that $E/F$ is totally ramified. Then, Lemma \ref{lemma_close} says that $E'/F'$ is also totally ramified and $E$ is $lm$-close to $E'$. Since $l\ne p$, the extension $E/F$ is tamely ramified, and we have
\begin{center}
	$E = \dfrac{F[T]}{(T^l-\varpi)}$ and 
	$E' = \dfrac{F'[T]}{(T^l-\varpi')}$.
\end{center}
From (\ref{Galois_Totally_ram}) of Lemma \ref{lemma_close}, we get the following isomorphism
$$ \Pi : \frac{A_m[T]}{(T^l-
\overline{\varpi})} \longrightarrow \frac{A_m'[T]}
{(T^l-\overline{\varpi'})}, $$
induced by the ring isomorphism $\Lambda:A_m\rightarrow A_m'$ 
and it is given by
$$ \Pi\Big(\sum_{i=0}^{l-1}\overline{x_i}T^i + 
(T^l-\overline{\varpi})\Big) = \sum_{i=0}^{l-1}
\Lambda(\overline{x_i})T^i + 
(T^l-\overline{\varpi'}), $$
where $\overline{x_i} \in A_m$, for each $i$. Let $\sigma$ and $\sigma'$ be the automorphisms of $E$ and $E'$ respectively, that sends the class of $T$ to the class of $T^2$, i.e.,
$$ \sigma(T + (T^l-\varpi)) 
= T^2 + (T^l-\varpi) $$
and 
$$ \sigma'(T + (T^l-\varpi')) = T^2 + 
(T^l-\varpi'), $$
Then $\sigma$ (resp. $\sigma'$) is a generator of $\Gamma$ (resp. $\Gamma'$). Moreover, the automorphism $\sigma$ induces a ring isomorphism $\mathfrak{o}_E/\mathfrak{p}_E^{em} \xrightarrow{\sim} \mathfrak{o}_E/\mathfrak{p}_E^{em}$, and the induced map is again denoted 
by $\sigma$. Similarly for $\sigma'$. Then we have
$$ \Pi \circ \sigma = \sigma' \circ \Pi. $$
\end{remark}
	
\subsection{Compatibility with Galois action}\label{Comp_Galois}
Here, we continue to work with the setup of the preceding subsection. We choose the generator $\sigma$ (resp. $\sigma'$) of $\Gamma$ (resp. $\Gamma'$), as discussed in Remark \ref{Galois_isom}. We have the following commutative diagram 
$$
\xymatrix{
	\mathfrak{o}_E/\mathfrak{p}^{em}_E \ar[dd]_{\Pi}  \ar[rr]^{\sigma}  &&
	\mathfrak{o}_E/\mathfrak{p}^{em}_E \ar[dd]^{\Pi} \\\\
	\mathfrak{o}_{E'}/\mathfrak{p}^{em}_{E'}  \ar[rr]_{\sigma'} &&
	\mathfrak{o}_{E'}/\mathfrak{p}^{em}_{E'}
}
$$
where $e=1$ if $E/F$ is unramified and $e=l$ if $E/F$ is totally ramified. To be more precise, for any $a\in\mathfrak{o}_E$, we have the relation
\begin{equation}\label{rel_1}
	\overline{\sigma(a)}'=\sigma(\overline{a})'=
	\sigma'(\overline{a}')=\overline{\sigma'(a')},
\end{equation} 
where $a'\in\mathfrak{o}_{E'}$ is the preimage of $\overline{a}'$ under ${\rm mod}-\mathfrak{p}_{E'}^{em}$. We fix an isomorphism $\iota:\Gamma\xrightarrow{\sim} \Gamma'$ that sends the generator $\sigma$ to $\sigma'$.
	
\subsection{}
Note that $\sigma$ can be regarded as an automorphism of the group $G_E$ via its natural action on $G_E$. Consider the compact open subgroup,
$$
K_E={\rm Ker}\big(\textbf{G}(\mathfrak{o}_E)\rightarrow\textbf{G}
(\mathfrak{o}_E/\mathfrak{p}_E^{em})\big).
$$ 
Since $K_E$ is a pro-$p$ subgroup of $G_E$, it follows from the arguments of \cite[Lemma 6.6]{feng2024smith} that $K_E$ is $\sigma$-plain. There 
is an induced action of $\langle\sigma\rangle$ on the Hecke algebra $\mathcal{H}(G_E,K_E)$, given by
$$
(\sigma.f)(x)=f(\sigma^{-1}x),
$$ 
for all $x\in G_E$ and $f \in \mathcal{H}(G_E,K_E)$. Similarly, 
there is an action of $\langle\sigma'\rangle$ on the Hecke algebra $\mathcal{H}(G_{E'},K_{E'})$.
We end this section with the following proposition, which shows the compatibility of Kazhdan isomorphism with these Galois actions.
\begin{proposition}\label{comp_galois_action}
For all $f\in\mathcal{H}(G_E,K_E)$, we have
$$
{\rm Kaz}_{em}^E(\sigma.f)=\sigma'.{\rm Kaz}_{em}^E(f).
$$
\end{proposition}
\begin{proof}
Because of the Cartan decomposition (\ref{MCD}), it is enough to show that
$$ {\rm Kaz}_{em}^E(\sigma.t_{a_i\pi_\mu a_j^{-1}})=\sigma'.{\rm Kaz}_{em}^E(t_{a_i\pi_\mu a_j^{-1}}),
$$ 
for all $\mu\in X_*(\textbf{T})^{-}$ and $(a_i,a_j)\in T_\mu(E)$.  
Note that
$$ \sigma.t_{a_i\pi_\mu a_j^{-1}}=
t_{\sigma(a_i)\sigma(\pi_\mu)\sigma(a_j^{-1})}. $$
By Cartan decomposition (\ref{MCD}), the
double coset space $K_E\sigma(a_i)\sigma(\pi_\mu)\sigma(a_j^{-1})K_E$ 
is of the form $K_Ec_i\pi_\mu c_j^{-1}K_E$, for some $(c_i,c_j)\in T_\mu(E)$. Then
\begin{equation}\label{1}
{\rm Kaz}_{em}^E(\sigma. t_{a_i\pi_\mu a_j^{-1}})={\rm Kaz}_{em}^E\big(t_{\sigma(a_i)\sigma(\pi_\mu)\sigma(a_j)
^{-1}}\big)={\rm Kaz}_{em}^E\big(t_{c_i\pi_\mu c_j^{-1}}\big)
=t_{c_i'\pi_\mu' c_j'^{-1}}.
\end{equation} 
Consider the Galois conjugate $\sigma(\pi)=\pi t=\pi\sigma(s)$, where $t,s\in\mathfrak{o}_E^\times$. Under the relation $\sigma(\mu(x))=\mu(\sigma(x))$, for $x\in E^\times$, we observe that
\begin{align*} 
K_Ec_i\pi_\mu c_j^{-1}K_E
&=K_E\sigma(a_i)\sigma(\pi_\mu)\sigma(a_j^{-1})K_E\\
&= K_E\sigma(a_i)\pi_\mu \sigma(s_\mu a_j^{-1})K_E\\
&= K_E\sigma(a_i)\pi_\mu \sigma(d_j^{-1})K_E,
\end{align*} 
where $s_\mu=\mu(s)$, and $d_j= a_js_\mu^{-1}\in\textbf{G}(\mathfrak{o}_E)$. Then $\big(\bar{c_i}^{-1}\overline{\sigma(a_i)},
\overline{\sigma(d_j^{-1})}\bar{c_j}\big)\in \Gamma_\mu(E)$, and using the
bijection $\Gamma_\mu(E)\rightarrow\Gamma_\mu(E')$, we get
$$ \big(\bar{c_i}'^{-1}\overline{\sigma(a_i)}',
\overline{\sigma(d_j^{-1})}'\bar{c_j}'\big)\in \Gamma_\mu(E').$$
Recall that the image of $d_j$ under the reduction map $\textbf{G}(\mathfrak{o}_E)\rightarrow
\textbf{G}(\mathfrak{o}_E/\mathfrak{p}_E^{em})$ is denoted by $\overline{d_j}$, and $\overline{d_j}'$ is the image of $\overline{d_j}$ under the isomorphism $\textbf{G}(\mathfrak{o}_E/\mathfrak{p}_E^{em})\xrightarrow{\sim}
\textbf{G}(\mathfrak{o}_{E'}/\mathfrak{p}_{E'}^{em})$. Let $d_j'\in\textbf{G}(\mathfrak{o}_{E'})$ be a preimage of
$\overline{d_j}'$ under ${\rm mod}-\mathfrak{p}_{E'}^{em}$. Then, the commutative diagram
$$
\xymatrix{
	\mathbb{G}_m(\mathfrak{o}_E/\mathfrak{p}^{em}_E) \ar[dd]_{\Pi}  \ar[rr]^{\mu}  &&
	\textbf{T}(\mathfrak{o}_E/\mathfrak{p}^{em}_E) \ar[dd]^{\Pi} \\\\
	\mathbb{G}_m(\mathfrak{o}_{E'}/\mathfrak{p}^{em}_{E'})  \ar[rr]_{\mu} &&
	\textbf{T}(\mathfrak{o}_{E'}/\mathfrak{p}^{em}_{E'})
}
$$
gives the relation 
$$
\overline{d_j'}=\overline{a_j'}\overline{s_\mu}'^{-1}
=\overline{a_j'}(\overline{s'_\mu})^{-1},
$$
where $s' \in \mathfrak{o}_{E'}$ is a preimage of $\overline{s}'$ under 
${\rm mod}-\mathfrak{p}_{E'}^{em}$. On the other hand, the relation $\sigma(\pi)=\pi\sigma(s)$ together with (\ref{rel_1}) implies that
$$
\overline{\sigma'(\pi')}=\sigma'(\overline{\pi}')
=\overline{\pi'}\overline{\sigma'(s')}.
$$
For any $\mu\in X_*(\textbf{T})^{-}$, the above equality gives
$$
\overline{(\sigma'(\pi'))_\mu}=\overline{(\pi'\sigma'(s'))_\mu}.
$$
This implies that $(\sigma'(\pi'))_\mu^{-1}(\pi'\sigma'(s'))_\mu \in K_{E'}$.
Now, using normality of $K_{E'}$ in $\textbf{G}(\mathfrak{o}_{E'})$ and the relation $\sigma(\mu(x))=\mu(\sigma(x))$, for $x\in E^\times$, we get
\begin{align*}
K_{E'}c_i'\pi_\mu'c_j'^{-1}K_{E'}
&=K_{E'}\sigma'(a_i')\pi_\mu' \sigma'(d_j'^{-1})K_{E'}\\
&=K_{E'}\sigma'(a_i') \sigma'(\pi_\mu') \sigma'(a_j'^{-1})K_{E'},
\end{align*}
which shows that
$$
t_{c_i'\pi_\mu' c_j'^{-1}}=t_{\sigma'(a_i')\sigma'(\pi_\mu')
\sigma'(a_j'^{-1})}=\sigma'.t_{a_i'\pi_\mu' a_j'^{-1}}.
$$
Finally, it follows from (\ref{1}) and the above relation that
$$
{\rm Kaz}_{em}^E(\sigma.t_{a_i\pi_\mu a_j^{-1}})=\sigma'.{\rm Kaz}_{em}^E(t_{a_i\pi_\mu a_j^{-1}}).
$$ 
\end{proof}
	
\section{Compatibility of Brauer homomorphism and Kazhdan
		isomorphism}\label{Br_1}
In this section, we prove Theorem \ref{intro_thm}. Let us first recall the notations and terminologies from the preceding sections. 
\subsection{}
Fix a positive integer $m$. Let $F$ and $F'$ be the non-Archimedean local fields such that $F$ is $m$-close to $F'$. Let $E$ be a finite Galois extension of $F$ of prime degree $l$ with $l\ne p$, where $p$ is the residue characteristic of both $F$ and $F'$. According to Lemma \ref{lemma_close}, there exists a finite Galois extension $E'$ of $F'$ of degree $l$ and $E'$ is $em$-close to $E$, where $e=1$ if $E/F$ is unramified and $e=l$ if $E/F$ is totally ramified. 
	
Let $\textbf{G}$ be a split connected reductive group defined over $\mathbb{Z}$. It follows from \cite[Proposition 12.9.4]{MR4520154} that the fixed point subgroup $K_E^\sigma$ (resp. $K_{E'}^{\sigma'}$) is the $m$-th congruence subgroup of $\textbf{G}(\mathfrak{o}_F)$ (resp. $\textbf{G}(\mathfrak{o}_{F'})$) i.e, $K_E^\sigma=K_F$ and $K_{E'}^{\sigma'}=K_{F'}$, where
$$
K_F={\rm Ker}\big(\textbf{G}(\mathfrak{o}_F)\rightarrow\textbf{G}
(\mathfrak{o}_F/\mathfrak{p}_F^m)\big)	
$$ and
$$
K_{F'}={\rm Ker}\big(\textbf{G}(\mathfrak{o}_{F'})\rightarrow\textbf{G}
(\mathfrak{o}_{F'}/\mathfrak{p}_{F'}^m)\big).
$$ 
Proposition \ref{comp_galois_action} then gives an isomorphism between $\mathcal{H}(G_E,K_E)^\sigma$ and $\mathcal{H}(G_{E'},K_{E'})^{\sigma'}$, induced by ${\rm Kaz}_{em}^E$. On the other hand, we have the Brauer homomorphism
$\overline{\rm Br}$ (resp. $\overline{\rm Br}'$) from the space $\mathcal{H}(G_E,K_E)^\sigma$ (resp. $\mathcal{H}(G_{E'},K_{E'})^{\sigma'}$) to the Hecke algebra $\mathcal{H}(G_F,K_F)$ (resp. $\mathcal{H}(G_{F'},K_{F'})$). With these set up, we aim to prove 
\begin{theorem}\label{main_thm}
Let $F, F', E$ and $E'$ be as above. Then, for any connected split reductive group $\textbf{G}$ defined over $\mathbb{Z}$, we have
$$ 
{\rm Kaz}_m^F\circ \overline{\rm Br} = \overline{\rm Br}'\circ {\rm Kaz}_{em}^E,
$$
where $e=1$ if $E/F$ is unramified and $e=l$ if $E/F$ is totally ramified.
\end{theorem}
\begin{proof}
We divide the proof into two parts. In the first part, we consider the extension $E/F$ unramified. The second part deals with the totally ramified case. To begin with, let us recall that an arbitrary $h \in\mathcal{H}(G_E,K_E)^\sigma$ can be written as
$$
h=\sum_{\mu\in X_*(\textbf{T})^{-}}\sum_{(a_i,a_j)\in T_\mu(E)}
\alpha_\mu(a_i,a_j) t_{a_i\pi_\mu a_j^{-1}},
$$
where each $\alpha_\mu(a_i,a_j) \in k$. Then  
\begin{equation}\label{eq_1}
({\rm Kaz}_m^F\circ {\rm \overline{Br}})(h)=\sum_{(a_i,a_j) 
\in T_\mu(E)} \alpha_\mu(a_i,a_j) ({\rm Kaz}_m^F\circ 
{\rm \overline{Br}})(t_{a_i\pi_\mu a_j^{-1}})
\end{equation}
and 
\begin{equation}\label{eq_2}
({\rm \overline{Br}}'\circ {\rm Kaz}_{em}^E)(h)=\sum_{\mu \in
X_*(\textbf{T})^{-}} \sum_{(a_i,a_j) \in T_\mu(E')} \alpha_\mu(a_i,a_j) {\rm \overline{Br}}'(t_{a_i'\pi_\mu' a_j'^{-1}}).
\end{equation} 
For each $\mu\in X_*(\textbf{T})^{-}$ and $(a_i,a_j)\in T_\mu(E)$, we have ${\rm \overline{Br}}(t_{a_i\pi_\mu a_j^{-1}})=1_{G_F\cap K_Ea_i\pi_\mu a_j^{-1}K_E}$. Then Cartan decomposition (\ref{MCD}) gives
\begin{equation}\label{decom_1}
K_Ea_i\pi_\mu a_j^{-1}K_E \cap G_F=\coprod_{\lambda \in X_*(\textbf{T})^{-}} \coprod_{(b_i,b_j) \in T_\lambda(F)} K_Ea_i\pi_\mu a_j^{-1}K_E \cap K_Fb_i\varpi_\lambda b_j^{-1}K_F.
\end{equation} Similarly, 
\begin{equation}\label{decom_2}
K_{E'}a_i'\pi_\mu'a_j'^{-1}K_{E'}\cap G_{F'}=\coprod_{\nu\in
X_*(\textbf{T})^{-}}\coprod_{(d_i',d_j')\in T_\nu(F')}
K_{E'}a_i'\pi'_\mu a_j'^{-1}K_{E'}\cap K_{F'}d_i'\varpi'_\nu
d_j'^{-1}K_{F'}.
\end{equation}
\subsubsection{Unramified case}
Suppose $E$ is unramified over $F$. Then the field $E'$ is also unramified over $F'$ and it is $m$-close to $E$ (Lemma \ref{lemma_close}). In this case, $\pi=\varpi$ and $\pi'=\varpi'$. For any $\lambda\in X_*(\textbf{T})^{-}$ and $(b_i,b_j)\in
T_\lambda(F)$, it follows from the
decompositions (\ref{decom_1}) and (\ref{decom_2}) that
$$ b_i\varpi_\lambda b_j^{-1}\in K_Ea_i\varpi_\mu a_j^{-1}K_E $$ 
if and only if 
$$ \varpi_\lambda \in K_Eb_i^{-1}a_i\varpi_\mu a_j^{-1}b_jK_E, $$ 
and this will happen only if $\lambda=\mu$. Then
$$ \varpi_\mu\in K_Eb_i^{-1}a_i\varpi_\mu a_j^{-1}b_jK_E $$ 
if and only if 
$$ (b_i^{-1}a_i, b_j^{-1}a_j)\in\Gamma_\mu(E). $$  
Using the bijection $\Gamma_\mu(E)
\rightarrow\Gamma_\mu(E')$ $(\overline{a}\mapsto\overline{a}')$, we have
$$ (\overline{b_i}^{-1}\overline{a_i}, \overline{b_j}^{-1}
\overline{a_j}) \in \Gamma_\mu(E) $$ 
if and only if
$$ (\overline{b_i}'^{-1}\overline{a_i}', \overline{b_j}'^{-1}\overline{a_j}')=
(\overline{b_i'}^{-1}\overline{a_i'}, \overline{b_j'}^{-1}\overline{a_j'})\in
\Gamma_\mu(E'), $$ 
where $(b_i',b_j')\in T_\mu(F')$. Therefore, we have
\begin{equation}\label{eq_3}
b_i\varpi_\mu b_j^{-1}\in K_Ea_i\varpi_\mu a_j^{-1}K_E \iff	b_i'\varpi_\mu'b_j'^{-1}\in
K_{E'}a_i'\varpi_\mu' a_j'^{-1}K_{E'}
\end{equation}
In view of (\ref{eq_3}), it follows from (\ref{eq_1}) and
(\ref{eq_2}) that
\begin{equation}\label{C_B}
({\rm Kaz}_m^F \circ {\rm \overline{Br}})(h) = ({\rm \overline{Br}}' \circ {\rm Kaz}_m^E)(h).
\end{equation}  
\subsubsection{Totally Ramified case}
Here, we assume that the extension $E/F$ is totally ramified. Then Lemma \ref{lemma_close} ensures that the extension $E'/F'$ is also totally ramified and $E'$ is $lm$-close to $E$, where $l=[E:F]=[E':F']$. 
Note that
$$ K_E={\rm Ker}\big(\textbf{G}(\mathfrak{o}_E)
\xrightarrow{{\rm mod}-\mathfrak{p}_E^{lm}}
\textbf{G}(\mathfrak{o}_E/\mathfrak{p}_E^{lm})\big) $$
and
$$ K_F={\rm Ker}\big(\textbf{G}(\mathfrak{o}_F)
\xrightarrow{\text{mod}-\mathfrak{p}_F^m}
\textbf{G}(\mathfrak{o}_F/\mathfrak{p}_F^m)\big). $$
Let $\xi\in X_*(\textbf{T})^{-}$ and $(c_i,c_j)\in T_\xi(F)$. In view of the decompositions (\ref{decom_1}) and (\ref{decom_2}), we have
$$ c_i\varpi_\xi c_j^{-1}\in K_Ea_i\pi_\mu a_j^{-1}K_E $$ if and only
if
$$ \varpi_\xi\in K_Ec_i^{-1}a_i\pi_\mu a_j^{-1}c_jK_E, $$ and it is
possible only if $l\xi=\mu$. Therefore, it follows that
$$ \pi_{l\xi}\in K_Ec_i^{-1}a_i\pi_{l\xi} a_j^{-1}c_jK_E $$ 
if and only if
$$ (\overline{c_i}^{-1}\overline{a_i},
\overline{a_j}^{-1}\overline{c_j})\in
\Gamma_{l\xi}(E). $$ 
Under the bijection $\Gamma_{l\xi}(E)\rightarrow \Gamma_{l\xi}(E')$, 
we have
$$ (\overline{c_i}^{-1}\overline{a_i},
\overline{a_j}^{-1}\overline{c_j})\in
\Gamma_{l\xi}(E') $$ if and only if
$$ (\overline{c_i}'^{-1}\overline{a_i}', \overline{c_j}'^{-1}\overline{a_j}')=
(\overline{c_i'}^{-1}\overline{a_i'}, \overline{c_j'}^{-1}\overline{a_j'}) \in
\Gamma_{l\xi}(E'),$$
where $(c_i',c_j')\in T_{l\xi}(F')$. Thus we have
\begin{equation}\label{equ_10}
c_i\varpi_{\xi} c_j^{-1}\in K_Ea_i\pi_{\mu} a_j^{-1}K_E \iff
c_i'\varpi'_{\xi}c_j'^{-1}\in K_{E'}a_i'\pi'_{\mu}
a_j'^{-1}K_{E'}.
\end{equation} 
From the relations (\ref{eq_1}), (\ref{eq_2}) and
(\ref{equ_10}), we get
$$
({\rm Kaz}_{m}^F\circ{\rm \overline{Br}})(h)=({\rm \overline{Br}}'
\circ{\rm Kaz}_{lm}^E)(h).
$$ 
This completes the proof.
\end{proof}
\section{Linkage over close local fields}
This section gives an application of Theorem \ref{main_thm} in the context of representation theory. In the article \cite{MR3432583}, Treumann--Venkatesh defined the notion of linkage, which is conjectured to be compatible with the Langlands functorial transfer. Using Theorem \ref{main_thm}, we show that linkage is compatible under close local fields. To prove this, we first formulate the compatibility of linkage with Brauer homomorphism, as described in \cite[Section 6.2]{MR3432583}. To begin with, we recall the notion of linkage and Tate cohomology. For precise reference, see \cite[Section 3 and 6]{MR3432583}.
\subsection{Tate cohomology}
Let $\Gamma$ be a cyclic group of order $n$ with a generator $\gamma$, and let $M$ be an abelian group with an action of $\Gamma$. Let $T_\gamma$ be the
automorphism of $M$ defined by
\begin{center}
$T_\gamma(m)=\gamma.m$,  for $\gamma\in\Gamma, m\in M$.
\end{center}
Let $N={\rm id}+T_\gamma+T_{\gamma^2}+\dots+T_{\gamma^{n-1}}$ be the
norm operator. The Tate cohomology groups $\widehat{H}^i(M)$, for $i\in\{0,1\}$, are defined as
\begin{center}
$\widehat{H}^0(M):={\rm Ker}({\rm id}-T_\gamma)/{\rm Im}(N)$,\\
$\widehat{H}^1(M):={\rm Ker}(N)/{\rm Im}({\rm id}-T_\gamma)$.   
\end{center}
\subsubsection{Tate Cohomology of sheaves on $l$-spaces}
Let $X$ be an $l$-space (i.e., locally compact, 
totally disconnected and Hausdorff) with an action of a finite group
$\langle\gamma\rangle$ of prime order $l$. Let $\mathcal{F}$ be a
sheaf of $k$-modules on $X$. If $\mathcal{F}$ is $\gamma$-equivariant,
then $\gamma$ can be regarded as a map of restricted sheaves
$\mathcal{F}|_{X^\gamma} \rightarrow \mathcal{F}|_{X^\gamma}$ and the
Tate cohomology is defined as
\begin{center}
$\widehat{H}^0(\mathcal{F}|_{X^\gamma}) := 
{\rm Ker}(1-\gamma)/{\rm Im}(N)$,\\
$\widehat{H}^1(\mathcal{F}|_{X^\gamma}) :=
{\rm Ker}(N)/{\rm Im}(1-\gamma)$.
\end{center}
For each $i$, the object
$\widehat{H}^i(\mathcal{F}|_{X^\gamma})$ is a sheaf on $X^\gamma$. If $\Gamma_c(X;\mathcal{F})$ denotes the space of compactly supported sections on $\mathcal{F}$, then we have the following result. 
\begin{proposition}\cite[Proposition 3.3.1]{MR3432583}\label{Tate_isom}
For each $i\in\{0,1\}$, the restriction map induces an isomorphism 
\begin{center}
$\widehat{H}^i(\Gamma_c(X;\mathcal{F}))\xrightarrow{\sim}
\Gamma_c(X^\gamma; \widehat{H}^i(\mathcal{F}))$.
\end{center}
\end{proposition}
	
\subsection{Linkage}\label{linkage_def}
Let $F$ be a non-Archimedean local field, and $G$ be the $F$-points of a connected reductive group defined over $F$. A representation $(\rho,V)$ of $G$ is called {\it smooth} if for every vector $v\in V$, the
$G$-stabilizer of $v$ is an open subgroup of $G$. The
representation $\rho$ is called {\it admissible} if, for each compact
open subgroup $K$ of $G$, the space of $K$-fixed vectors $V^K$ (sometimes denoted by $\rho^K$) is
finite-dimensional over $k$. In this subsection, all the
representations are assumed to be admissible, and the representation
spaces are defined over $k$. We denote by $\mathcal{R}(G)$ the category of smooth, admissible, $k$-representations of $G$. 
\subsubsection{Frobenius Twist}
Let $(\rho,V)$ be a representation of $G$. Consider the vector space $V^{(l)}$, where the underlying additive group structure of $V^{(l)}$ is same as that of $V$ but the scalar action $\star$ on $V^{(l)}$ is given by
\begin{center}
$c\star v=c^{\frac{1}{l}}.v$ , for all $c\in k, v\in V$.
\end{center}
Then, the action of $G$ on $V$ induces a smooth representation $\rho^{(l)}$ of $G$ on the space $V^{(l)}$. The representation $(\rho^{(l)},V^{(l)})$ is called the {\it Frobenius twist} of $(\rho,V)$.
\subsubsection{}
Let $\sigma$ be an automorphism of $G$ of prime order $l$. An irreducible representation $(\Xi,W)$ of $G$ is called $\sigma$-fixed if $\Xi$ is isomorphic to the twisted representation $\Xi\circ\sigma$. In such a case, there is a unique action of $\sigma$ on $W$, compatible with the action of $\sigma$ on $G$ (see \cite[Proposition 6.1]{MR3432583}). Then the Tate cohomology groups $\widehat{H}^i(W)$, for $i\in\{0,1\}$, are defined as $k$-representations of $G^\sigma$. We denote these representations by $\widehat{H}^i(\Xi)$.
\begin{definition}
\rm An irreducible representation $\rho$ of $G^\sigma$ is linked
with the representation $\Xi$ if the Frobenius twist $\rho^{(l)}$
occurs as a Jordan-Holder constituent of $\widehat{H}^0(\Xi)$ or
$\widehat{H}^1(\Xi)$.
\end{definition}
In \cite[Section 6.3]{MR3432583}, the authors made a conjecture that the Tate cohomology groups $\widehat{H}^i(\Xi)$, $i\in\{0,1\}$, are of finite length as representations of $G^\sigma$. We give a proof of this conjecture for the general linear group ${\rm GL}_n$ in the local base change setting, using Bernstein--Zelevinksy derivatives.
	
\subsection{Finiteness of Tate cohomology}
Before going into the main theorem, we first recall the notion of derivatives for smooth representations of ${\rm GL}_n(L)$, where $L$ is a non-Archimedean local field. We refer to \cite[Chapter III]{Vigneras_mod_l} and \cite[Section 3]{BZ_induced} for more details. Let $P_n(L)$ and $U_n(L)$ be the subgroups of ${\rm GL}_n(L)$, given by
$$ P_n(L) = \Big\{\begin{pmatrix}
	A & B\\
	0 & 1
\end{pmatrix}: A\in {\rm GL}_{n-1}(L), B\in L^{n-1}\Big\} $$ 
and
$$ U_n(L)=\Big\{\begin{pmatrix}
	1 & C\\
	0 & 1
\end{pmatrix}: C\in L^{n-1}\Big\} $$ 
respectively. We use the short notation $X_{L,n}$ to denote the coset space $P_n(L)/P_{n-1}(L)U_n(L)$. Let $\psi_L:L\rightarrow k^\times$ be a non-trivial additive character. Let $N_n(L)$ be the group of all unipotent upper triangular matrices in ${\rm GL}_n(L)$. Let $\Theta_L$ be the character of $N_n(L)$, defined as
\begin{equation}\label{non-deg_char}
\Theta_L((a_{ij})_{i,j=1}^n) : = \psi_L(a_{12}+a_{23}+\cdots a_{n-1,n}),
\end{equation}
for $(a_{ij})_{i,j=1}^n\in N_n(L)$. By abuse of notation, the restriction ${\rm res}_{U_n(L)}\Theta_L$ is also denoted by $\Theta_L$. We have four fundamental functors:
$$
\Psi^-:\mathcal{R}(P_n(L))\rightarrow\mathcal{R}({\rm GL}_{n-1}(L)), \Psi^+:\mathcal{R}({\rm GL}_{n-1}(L))\rightarrow\mathcal{R}(P_n(L))
$$
$$
\Phi^-:\mathcal{R}(P_n(L))\rightarrow \mathcal{R}(P_{n-1}(L)), \Phi^+:\mathcal{R}(P_{n-1}(L))\rightarrow\mathcal{R}(P_n(L)),
$$
given by $\Psi^-=r_{U_n(L),{\rm id}}$, $\Phi^-=r_{U_n(L),\Theta_L}$, $\Psi^+=i_{U_n(L),{\rm id}}$, and $\Phi^+=i_{U_n(L),\Theta_L}$. 
	
\subsubsection{}
Let $\tau$ be a smooth representation of $P_n(L)$. The $m$-th derivative of $\tau$, denoted by $\tau^{(m)}$, is constructed as the representation $\Psi^-(\Phi^-)^{m-1}(\tau)$ of ${\rm GL}_{n-m}(L)$. Moreover, we get a functorial filtration on $\tau$:
$$
0\subseteq\tau_{n}\subseteq\tau_{n-1}\subseteq\dots
\subseteq\tau_2\subseteq\tau_1=\tau,
$$
where $\tau_m=(\Phi^+)^{m-1}(\Phi^-)^{m-1}(\tau)$ and $\tau_m/\tau_{m+1}=(\Phi^+)^{m-1}(\Psi^+)(\tau^{(m)})$. We now deduce:
\begin{lemma}\label{finite_length}
Let $\rho$ be a finite length representation of $G_t(L)$, where $1\leq t < n$. Then, the $P_n(L)$-representation 
$(\Phi^+)^{n-t-1}(\Psi^+)(\rho)$ 
is also of finite length.
\end{lemma}
\begin{proof}
This is an immediate consequence of \cite[Chapter III, Subsection 1.5]{Vigneras_mod_l} and the exactness of the functor $(\Phi^+)^{n-t-1}(\Psi^+)$.
\end{proof}
	
\subsubsection{}
Let $F$ be a non-Archimedean local field with residue characteristic $p$. Let $E/F$ be a finite Galois extension of prime degree $l$ with $l\ne p$. Let $\sigma$ be a generator of ${\rm Gal}(E/F)$. Fix a non-trivial additive character $\psi_F:F\rightarrow k^\times$. Let $\psi_E$ be the character $\psi_F\circ{\rm Tr}_{E/F}$, where ${\rm Tr}_{E/F}$ is the trace function. We have the non-degenerate characters $\Theta_F$ and $\Theta_E$ of $N_n(F)$ and $N_n(E)$ respectively, defined by (\ref{non-deg_char}). We now prove the following finiteness result of Tate cohomology:
\begin{proposition}\label{finiteness}
Let $\Xi$ be an irreducible $\sigma$-fixed representation of ${\rm GL}_n(E)$. Let $T_\sigma:\Xi\rightarrow \Xi\circ\sigma$ be an isomorphism with $T_\sigma^l={\rm id}$. Then the Tate cohomology group $\widehat{H}^i(\Xi)$, with respect to the operator $T_\sigma$, has finite length as a representation of ${\rm GL}_n(F)$.
\end{proposition}
\begin{proof}
We prove the proposition using induction on $n$. The case $n=1$ is clear. So, we assume that the proposition is true for all irreducible $\sigma$-fixed representations of ${\rm GL}_t(E)$ and for all $t<n$. Now, we consider $\Xi$ as a representation of $P_n(E)$. Since $\psi_E(\sigma(x))=\psi_E(x)$, for $x\in E$, we get the isomorphism
\begin{equation}\label{isom_1}
\Phi^-(\Xi\circ\sigma)\simeq\Phi^-(\Xi)\circ\sigma,
\end{equation}
as representation of $P_{n-1}(E)$. Similarly, for any representation $\tau\in\mathcal{R}(P_{n-1}(E))$, we have the $P_n(E)$-equivariant isomorphism 
\begin{equation}\label{isom_2}
\Phi^+(\tau\circ\sigma) \simeq \Phi^+(\tau)\circ\sigma.
\end{equation}
Using (\ref{isom_1}) and (\ref{isom_2}), and the isomorphism $T_\sigma$, we get an isomorphism between the representations $\Xi_m$ and $\Xi_m\circ\sigma$, and also between the representations $\Xi^{(m)}$ and $\Xi^{(m)}\circ\sigma$, for all $m\leq n$. 
		
Recall that $X_{E,m}$ denotes the coset space $P_m(E)/P_{m-1}(E)U_m(E)$. Since $P_{m-1}(E)U_m(E)$ is a $\sigma$-stable subgroup of $P_m(E)$, we have the following long exact sequence of non-abelian cohomology (\cite[Appendix, Proposition 1]{Serre-local-field}):
\begin{equation}\label{ses_1}
0\longrightarrow P_{m-1}(F)U_m(F)\longrightarrow P_m(F)\longrightarrow X_{E,m}^\sigma\longrightarrow H^1(\sigma, P_{m-1}(E)U_m(E))\longrightarrow H^1(\sigma, P_m(E))
\end{equation}
Consider the short exact sequence of non-abelian ${\rm Gal}(E/F)$-modules
\begin{equation}\label{ses_2}
0\longrightarrow U_m(E)\longrightarrow P_{m-1}(E)U_m(E)\longrightarrow P_{m-1}(E)\longrightarrow 0
\end{equation} 
Using Hilbert's theorem 90, we get that the pointed sets $H^1(\sigma, U_m(E))$ and $H^1(\sigma, P_{m-1}(E))$ are trivial. Then, it follows from the long exact sequence of non-abelian cohomology corresponding to (\ref{ses_2}) that $H^1(\sigma, P_{m-1}(E)U_m(E)) = 0$. Hence, the long exact sequence (\ref{ses_1}) gives the equality $X_{E,m}^\sigma = X_{F,m}$. Now, using Proposition \ref{Tate_isom} repeatedly $(m-1)$-times, we get the $P_n(F)$-equivariant isomorphism
$$
\widehat{H}^i(\Xi_m/\Xi_{m+1}) \simeq
(\Phi^+)^{m-1}(\Psi^+)\big(\widehat{H}^i(\Xi^{(m)})\big).
$$
By induction hypothesis, the ${\rm GL}_{n-m}(F)$-representation $\widehat{H}^i(\Xi^{(m)})$ is of finite length, for all $m<n$. In view of Lemma \ref{finite_length}, it then follows from the above isomorphism that the $P_n(F)$-representation $\widehat{H}^i(\Xi_m/\Xi_{m+1})$ is of finite length, for all $m<n$. 
		
Now, for each $m\in\{1,2,\dots,n-1\}$, consider the short exact sequence of $P_n(E)$-representations
$$
0\longrightarrow \Xi_{m+1}\longrightarrow \Xi_m \longrightarrow \Xi_m/\Xi_{m+1} \longrightarrow 0.
$$ 
Since ${\rm Gal}(E/F)$ is cyclic, the corresponding long exact sequence of Tate cohomology gives the following diagram:
$$
\begin{tikzcd}[column sep={1.5cm,between origins}, row sep={1.932050908cm,between origins}]
& \widehat{H}^0(\Xi_{m+1}) \arrow[rr,]  && \widehat{H}^0(\Xi_m) \arrow[rd,] &  \\
\widehat{H}^1(\Xi_m/\Xi_{m+1}) \arrow[ru, ]&  &&  & \widehat{H}^0(\Xi_m/\Xi_{m+1}) \arrow[ld,] \\
& \widehat{H}^1(\Xi_m) \arrow[lu, ] && \widehat{H}^1(\Xi_{m+1}) \arrow[ll, ] & 
\end{tikzcd}
$$
We denote the above exact sequence by $S(m)$. Recall that $\Xi_n$ is the compactly induced representation ${\rm ind}_{N_n(E)}^{P_n(E)}(\Theta_E)$, and using \cite[Proposition 3.3.1]{MR3432583}, we get that $\widehat{H}^0(\Xi_n)$ is isomorphic to the irreducible representation ${\rm ind}_{N_n(F)}^{P_n(F)}(\Theta_F^l)$ and $\widehat{H}^1(\Xi_n) = 0$. In particular, the representation $\widehat{H}^i(\Xi_n)$, for $i\in\{0,1\}$, is of finite length. Now, using the exact sequence $S(n-1)$ and the finiteness of $\widehat{H}^i(\Xi_{n-1}/\Xi_n)$, we get that $\widehat{H}^i(\Xi_{n-1})$ is a finite length representation of $P_n(F)$. Again, using the finiteness of both the representations $\widehat{H}^i(\Xi_{n-1})$ and $\widehat{H}^i(\Xi_{n-2}/\Xi_{n-1})$, it follows from the exact sequence $S(n-2)$ that the representation $\widehat{H}^i(\Xi_{n-2})$ is of finite length. Thus, after a finite number of similar inductive steps, we finally get that $\widehat{H}^i(\Xi)$ is of finite length as a representation of $P_n(F)$, and hence of ${\rm GL}_n(F)$. This completes the proof.
\end{proof}
It is reasonable to expect that Proposition \ref{finiteness} holds for arbitrary connected reductive algebraic groups $\textbf{G}$ over $F$. Next, we give a formulation of the compatibility of Brauer homomorphism with linkage under the finite length assumption of the Tate cohomology $\widehat{H}^i(\Xi)$ for any irreducible $\sigma$-fixed representation $\Xi$ of $G=\textbf{G}(F)$. It enables us to study linkage from the module theoretic point of view via Brauer homomorphism. Then, using Theorem \ref{main_thm}, we show that linkage is compatible with the Kazhdan isomorphism.
	
\subsection{Compatibility of Brauer homomorphism with linkage}\label{com_linkage}
Let $\Xi$ be an irreducible $\sigma$-fixed representation of $G$, and
let $T:\Xi\rightarrow \Xi\circ\sigma$ be an isomorphism such that
$T^l={\rm id}$. For any $\sigma$-plain compact open subgroup $K$ of
$G$, the inclusion map $\Xi^K\hookrightarrow \Xi$ is
$\sigma$-invariant and it induces a map
$\widehat{H}^i(\Xi^K)\rightarrow \widehat{H}^i(\Xi)$, taking values in
the $\mathcal{H}(G^\sigma,K^\sigma)$-module
$\widehat{H}^i(\Xi)^{K^\sigma}$. It is shown in \cite[Section
6.2]{MR3432583} that, for all $h\in\mathcal{H}(G,K)^\sigma$, the Brauer homomorphism $\overline{\rm Br}$ induces the following commutative diagram
$$
\xymatrix{
	\widehat{H}^i(\Xi^K)\ar[dd]_{\widehat{H}^i(h)}  \ar[rr]^{}  &&
	\widehat{H}^i(\Xi)^{K^\sigma} \ar[dd]^{{\rm \overline{Br}}(h)} \\\\
	\widehat{H}^i(\Xi^K) \ar[rr]_{} &&  \widehat{H}^i(\Xi)^{K^\sigma}
}
$$
and it induces the compatibility of Brauer homomorphism with linkage in the sense that, for a representation $\rho$ of $G^\sigma$ to be linked with $\Xi$, it is equivalent to say that there exists a composition factor $\mathcal{N}$ of the $\mathcal{H}(G,K)^\sigma$-module $\widehat{H}^i(\Xi^K)$ such that
\begin{equation}\label{formulation_Br_linkage}
\mathcal{H}(G^\sigma,K^\sigma)\otimes_{\mathcal{H}(G,K)^\sigma} \mathcal{N} 
\simeq (\rho^{(l)})^{K^\sigma}.
\end{equation}
	
\subsubsection{Compatibility of Kazhdan isomorphism with linkage}
Let $F$ and $F'$ be two non-Archimedean $m$-close local fields with the same residue characteristic $p$. Let $E$ be a finite Galois extension of $F$ of prime degree $l$ with $l\ne p$. By Lemma \ref{lemma_close}, we have a finite Galois extension $E'/F'$ of degree $l$ such that $E'$ is $em$-close to $E$, where $e=1$ if $E/F$ is unramified and $e=l$ if $E/F$ is totally ramified. Fix two generators $\sigma$ and $\sigma'$ of the groups ${\rm Gal}(E/F)$ and ${\rm Gal}(E'/F')$ respectively.  
	
Let $\textbf{G}$ be a split connected reductive group defined over $\mathbb{Z}$. Let $K_E$ (resp. $K_F$) be the congruence subgroup of $G_E$ (resp. $G_F$) of level $m$ (resp. $em$) (see Section \ref{K} for the definition). Then, from Theorem \ref{main_thm}, we get the following relation:
\begin{equation}\label{BK}
{\rm Kaz}_m^F\circ {\rm \overline{Br}}={\rm \overline{Br}'}\circ 
{\rm Kaz}_{em}^E.
\end{equation}
Moreover, we have the following bijections (see \cite[Section 2.3]{ganapathy2015local})
\begin{equation}\label{ms_1}
\big\{(\rho_E,V_E)\in {\rm Irr}(G_E):V_E^{K_E}\ne 0\big\}\longrightarrow\big\{(\rho_{E'}, V_{E'})\in {\rm
Irr}(G_{E'}):V_{E'}^{K_{E'}}\ne 0\big\}
\end{equation} and
\begin{equation}\label{ms_2}
\big\{(\rho_F,V_F)\in {\rm Irr}(G_F):V_F^{K_F}\ne 0\big\}
\longrightarrow\big\{(\rho_{F'},V_{F'})\in {\rm Irr}(G_{F'}):
V_{F'}^{K_{F'}}\ne 0\big\},
\end{equation}
induced by the isomorphisms ${\rm Kaz}_{em}^E$ and ${\rm Kaz}_m^F$, respectively. With these notations, we now prove
\begin{theorem}\label{linkage_kaz}
Let $(\rho_E,V_E)$ be an irreducible $\sigma$-fixed representation of $G_E$ with $V_E^{K_E}\ne 0$, and let $(\rho_F,V_F)$ be an irreducible
representation of $G_F$ with $V_F^{K_F}\ne 0$. Let $(\rho_{E'},V_{E'})$ and $(\rho_{F'},V_{F'})$ be the corresponding objects under the bijections (\ref{ms_1}) and (\ref{ms_2}). Then, the irreducible representation $\rho_{E'}$ of $G_{E'}$ is $\sigma'$-fixed, and $\rho_F'$ is linked with $\rho_E'$ if and only if $\rho_F$ is linked with $\rho_E$.
\end{theorem}
\begin{proof}
We divide the proof into two parts. The first part proves that the irreducible representation $\rho_{E'}$ is $\sigma'$-fixed. In the second part, we show that the linkage between $\rho_F$ and $\rho_E$ implies the linkage between $\rho_{F'}$ and $\rho_{E'}$, and vice versa.
\subsubsection{}
Recall that the space of $K_E$-fixed vectors $V_E^{K_E}$ is a simple $\mathcal{H}(G_E,K_E)$-module, and it can be regarded as a $\mathcal{H}(G_{E'},K_{E'})$-module via the Kazhdan isomorphism ${\rm Kaz}_{em}^E$. Consider the $\mathcal{H}(G_{E'})$-module
$$
\mathcal{N}'= \mathcal{H}(G_{E'})\otimes_{\mathcal{H}(G_{E'},K_{E'})}V_E^{K_E}.
$$ 
Note that the action of $\langle\sigma'\rangle$ on $\mathcal{H}(G_{E'})$ induces an action of $\langle\sigma'\rangle$ on $\mathcal{N}'$, by setting
$$
\sigma'.(f'\otimes v)=(\sigma'.f')\otimes v,
$$ 
for all $f'\in\mathcal{H}(G_{E'})$ and $v\in V_E^{K_E}$. Following the construction of \cite[Chapter I, Subsection 4.4]{Vigneras_mod_l}, the $\mathcal{H}(G_{E'})$-module $\mathcal{N}'$ gives the irreducible representation $(\rho_{E'},V_{E'})$, defined as
$$
\rho_{E'}(g')(x')=1_{g'K_{E'}}*x',
$$ 
for $g'\in G_{E'}, x'\in V_{E'}$. Here $*$ denotes the action of the Hecke algebra $\mathcal{H}(G_{E'})$ on the space $V_{E'}$ and $1_{g'K_{E'}}$ is the characteristic function of the left coset $g'K_{E'}$. Since $\rho_E$ is $\sigma$-fixed, there exists an isomorphism $T_\sigma:\rho_E \rightarrow \rho_E\circ\sigma$. Then the map $T_{\sigma'}:\mathcal{N}'\rightarrow \mathcal{N}'$, defined by 
$$
T_{\sigma'}(f'\otimes v) := (\sigma'.f')\otimes T_\sigma(v), 
$$
induces an isomorphism between $\rho_{E'}$ and $\rho_{E'}\circ\sigma'$. Hence, $\rho_{E'}$ is $\sigma'$-fixed.
		
\subsubsection{}
Suppose the representation $\rho_F$ is linked with $\rho_E$. Then, the space of $K_F$-fixed vectors $(V_F^{(l)})^{K_F}$, considered as $\mathcal{H}(G_F,K_F)$-module, is a composition factor of $\widehat{H}^i(V_E)^{K_F}$, for some $i$. If we regard these as $\mathcal{H}(G_E,K_E)^\sigma$-modules via the Brauer homomorphism $\overline{\rm Br}$, then the formulation (\ref{formulation_Br_linkage}) gives a composition factor $\mathcal{M}$ of $\widehat{H}^i(V_E^{K_E})$ such that 
\begin{equation}\label{iso_1}
\mathcal{H}(G_F,K_F)\otimes_{\mathcal{H}(G_E,K_E)^\sigma} \mathcal{M}\simeq (V_{F}^{(l)})^{K_F}.
\end{equation} 
Note that the isomorphism $V_E^{K_E} \rightarrow V_{E'}^{K_{E'}}$ of $\mathcal{H}(G_E,K_E)$-modules induces a $\mathcal{H}(G_E,K_E)^\sigma$-module isomorphism $\Phi:\widehat{H}^i(V_E^{K_E}) \rightarrow \widehat
{H}^i(V_{E'}^{K_{E'}})$, which is also a $\mathcal{H}(G_{E'},K_{E'})^{\sigma'}$-module isomorphism via the Kazhdan map ${\rm Kaz}^E_{em}$. 
		
Let $\mathcal{M}'$ be the image of $\mathcal{M}$ under $\Phi$. Then $\mathcal{M}'$ is a composition factor of $\widehat{H}^i(V_{E'}^{K_{E'}})$, and the $\mathcal{H}(G_{F'},K_{F'})$-module 
$$
\mathcal{H}(G_{F'},K_{F'})\otimes_
{\mathcal{H}(G_{E'},K_{E'})^{\sigma'}} \mathcal{M}'
$$ 
is a composition factor of $\widehat{H}^i(V_{E'})^{K_{F'}}$. We also have an isomorphism $\varphi:V_F^{K_F} \rightarrow V_{F'}^{K_{F'}}$, which is equivariant under the action of $\mathcal{H}(G_F,K_F)$ and hence of $\mathcal{H}(G_{F'},K_{F'})$ via Kazhdan isomorphism ${\rm Kaz}_m^F$. Using the isomorphisms $\Phi$, $\varphi$ and the relation (\ref{BK}), it follows from (\ref{iso_1}) that
$$
\mathcal{H}(G_{F,},K_{F'})\otimes_{\mathcal{H}(G_{E'},K_{E'})^{\sigma'}} \mathcal{M}'\simeq (V_{F'}^{(l)})^{K_{F'}}.
$$
This implies that the representation $\rho_{F'}$ is linked with $\rho_{E'}$. A similar argument also shows that the representation $\rho_F$ is linked with $\rho_E$ if $\rho_{F'}$ is linked with $\rho_{E'}$. Hence, the theorem.
\end{proof}
	
\textbf{Acknowledgements.} The author expresses his deep gratitude to the anonymous referee for careful reading of the manuscript and for various suggestions. The author is grateful to Santosh Nadimpalli for communicating Proposition \ref{finiteness}, and for his helpful comments during the work. The author would also like to thank Radhika Ganapathy for her valuable comments and suggestions.

\bibliographystyle{amsalpha}
\bibliography{BK23.bib}	
	
\noindent\\
Department of Mathematics and Statistics,\\
Indian Institute of Technology Kanpur,
U.P. 208016, India.\\
Email: \texttt{mathsabya93@gmail.com}, \texttt{sabya@iitk.ac.in}

\end{document}